\documentclass[12pt]{amsart}
\usepackage{latexsym}
\usepackage{amssymb}
\usepackage{amscd}
\usepackage{mathrsfs}
\usepackage{rotating}
\renewcommand\a{\alpha}
\renewcommand\b{\beta}

\renewcommand\d{\delta}
\newcommand\la{\lambda}
\newcommand\z{\zeta}
\newcommand\e{\eta}
\renewcommand\th{\theta}

\newcommand\s{\sigma}
\newcommand\x{\chi}
\newcommand\f{\phi}
\newcommand\vf{\varphi}

\renewcommand\r{\rho}

\newcommand\vL{\varLambda}

\newcommand{\vT}{\varTheta}

\newcommand\ve{\varepsilon}

\newcommand\Fq{{\mathbf F}_q}

\newcommand\Ql{\bar{\mathbf Q}_l}

\newcommand\BQ{\mathbf Q}

\newcommand\BC{\mathbf C}

\newcommand\BZ{\mathbf Z}

\newcommand\Bm{\mathbf m}

\newcommand\Bv{\mathbf v}
\newcommand\Bk{\mathbf k}

\newcommand\Bla{\boldsymbol\lambda}

\newcommand\Bmu{\boldsymbol\mu}
\newcommand\Bnu{\boldsymbol\nu}

\newcommand\Bxi{\boldsymbol{\xi}}

\newcommand\CP{\mathcal{P}}

\newcommand\CX{ \mathcal{X}}

\newcommand\SG{\mathscr{G}}
\newcommand\SL{\mathscr{L}}

\newcommand\SO{\mathscr{O}}
\newcommand\SP{\mathscr{P}}
\newcommand\SQ{\mathscr{Q}}
\newcommand\SH{\mathscr{H}}

\newcommand\SX{\mathscr{X}}

\newcommand\Fo{\mathfrak o}

\newcommand\iv{^{-1}}
\newcommand\wh{\widehat}
\newcommand\wt{\widetilde}
\newcommand\wg{^{\wedge}}

\newcommand\ol{\overline}

\newcommand\trleq{\trianglelefteq}
\newcommand\trreq{\trianglerighteq}

\newcommand\hra{\hookrightarrow}

\newcommand\lv{\prec}
\newcommand\gv{\succ}

\newcommand\IC{\operatorname{IC}}

\newcommand\Hom{\operatorname{Hom}}
\newcommand\End{\operatorname{End}}

\newcommand\reg{_{\operatorname{reg}}}

\newcommand\unip{\operatorname{uni}}
\newcommand\uni{_{\operatorname{uni}}}

\newcommand\lp{\operatorname{\!\langle\!}}
\newcommand\rp{\operatorname{\!\rangle\!}}
\renewcommand\Im{\operatorname{Im}}

\newcommand{\isom}{\,\raise2pt\hbox{$\underrightarrow{\sim}$}\,}
\numberwithin{equation}{section}

\newtheorem{thm}{Theorem}[section]
\newtheorem{lem}[thm]{Lemma}
\newtheorem{cor}[thm]{Corollary}
\newtheorem{prop}[thm]{Proposition}

\def \para#1{\par\medskip\textbf{#1}
              \addtocounter{thm}{1}}

\def \remark#1{\par\medskip\noindent
                \textbf{Remark #1}
                \addtocounter{thm}{1}}

\begin{document}
\setlength{\baselineskip}{4.9mm}
\setlength{\abovedisplayskip}{4.5mm}
\setlength{\belowdisplayskip}{4.5mm}
%%%
%%%
\renewcommand{\theenumi}{\roman{enumi}}
\renewcommand{\labelenumi}{(\theenumi)}
\renewcommand{\thefootnote}{\fnsymbol{footnote}}
%%%
\renewcommand{\thefootnote}{\fnsymbol{footnote}}
%%%
\allowdisplaybreaks[2]
%\NoBlackBoxes
\parindent=20pt
%\addtocounter{section}{1}

%%%%%%%%%%%%%%%%%%%%
%%%%%%%%%%%%%%%%%%%%%%%%%%%%%%%%%%%
\pagestyle{myheadings}
\medskip
\begin{center}
 {\bf Kostka functions associated to complex reflection groups} 
\end{center}

\par\bigskip

\begin{center}
Toshiaki Shoji
\\  
\vspace{0.5cm}

%\vspace{0.8cm}

%{\it To the memory of J. A. Green}

%\vspace{0.7cm}

\end{center}

\title{}

\begin{abstract}
Kostka functions $K^{\pm}_{\Bla, \Bmu}(t)$ associated to complex reflection groups are 
a generalization of Kostka polynomials, which are indexed by a pair $\Bla, \Bmu$
of $r$-partitions of $n$ (and by the sign $+, -$).  It is expected that there exists a close relationship
between those Kostka functions and the intersection cohomology associated 
to the enhanced variety $\SX$ of level $r$.
In this paper, we study combinatorial properties of $K^{\pm}_{\Bla,\Bmu}(t)$ based on 
the geometry of $\SX$.  
In paticular, we show that in the case where $\Bmu = (-,\dots, -,\mu^{(r)})$
(and for arbitrary $\Bla$), $K^-_{\Bla, \Bmu}(t)$ has a Lascoux-Sch\"utzenberger 
type combinatorial description. 
\end{abstract}

\maketitle
%\markboth{SHOJI}{KOSTKA FUNCTIONS}
\pagestyle{myheadings}

\begin{center}
{\sc Introduction}
\end{center}
\par\bigskip
In 1981, Lusztig gave a geometric interpretation of Kostka polynomials 
in the following sense; let $V$ be an $n$-dimensional vector space 
over an algebraically closed field, and put $G = GL(V)$. Let $\SP_n$ 
be the set of partitions of $n$.  Let $\SO_{\la}$ be the 
unipotent class in $G$ labelled by $\la \in \SP_n$, and $K = \IC(\ol \SO_{\la}, \Ql)$
the intersection cohomology associated to the closure $\ol\SO_{\la}$ of $\SO_{\la}$.
Let $K_{\la,\mu}(t)$ be the Kostka polynomial indexed by $\la, \mu \in \SP_n$, and 
$\wt K_{\la,\mu}(t) = t^{n(\mu)}K_{\la,\mu}(t\iv)$ the modified Kostka polynomial 
(see 1.1 for the definition $n(\mu)$). Lusztig proved that 

\begin{equation*}
\tag{0.1}
\wt K_{\la,\mu}(t) = t^{n(\la)}\sum_{i \ge 0}\dim (\SH^{2i}_xK)t^i
\end{equation*}
for $x \in \SO_{\mu} \subset \ol\SO_{\la}$, where $\SH^{2i}_xK$ is the stalk at 
$x$ of the $2i$-th cohomology sheaf $\SH^{2i}K$ of $K$. 
(0.1) implies that $K_{\la,\mu}(t) \in \BZ_{\ge 0}[t]$. 
\par
Let $\SP_{n,r}$ be the set of $r$-tuple of partitions 
$\Bla =  (\la^{(1)}, \dots, \la^{(r)})$ such that $\sum_{i=1}^r |\la^{(i)}| = n$
(we write $|\la^{(i)}| = m$ if $\la^{(i)} \in \SP_m$).
In [S1], [S2], Kostka functions $K^{\pm}_{\Bla,\Bmu}(t)$ 
associated to complex reflections groups (depending on the signs $+, -$) are introduced, 
which are apriori rational functions in $t$ indexed by $\Bla, \Bmu \in \SP_{n,r}$.    
In the case where $r = 2$ (in this case $K^-_{\Bla,\Bmu}(t) = K^+_{\Bla, \Bmu}(t)$), 
it is proved in [S2] that $K^{\pm}_{\Bla, \Bmu}(t) \in \BZ[t]$. 
In this case, Achar-Henderson [AH] proved that those (generalized) Kostka polynomials 
have a geometric interpretation in the following sense; under the previous notation, 
consider the variety $\SX = G \times V$ on which $G$ acts naturally. Put 
$\SX\uni = G\uni \times V$, where $G\uni$ is the set of unipotent elements in $G$. 
$\SX\uni$ is a $G$-stable subset of $\SX$, and is isomorphic 
to the enhanced nilpotent cone introduced by [AH]. 
It is known by [AH], [T] that $\SX\uni$ has finitely many $G$-orbits, which are 
naturally parametrized by $\SP_{n,2}$. 
They proved in [AH] that the modified Kostka polynomial $\wt K^{\pm}_{\Bla, \Bmu}(t)$
$(\Bla, \Bmu \in \SP_{n,2}$), defined 
in a similar way as in the original case, 
can be written as in (0.1) in terms of the intersection cohomology 
associated to the closure $\ol\SO_{\Bla}$ of the $G$-orbit $\SO_{\Bla} \subset \SX\uni$.
\par
In the case where $r = 2$, the interaction of geometric properties and combinatorial 
properties of Kostka polynomials was studied in [LS]. 
In particular, it was proved that in the special case where $\Bmu = (-,\mu^{(2)})$ 
(and for arbitrary $\Bla \in \SP_{n,2}$), $K_{\Bla,\Bmu}(t)$ has a combinatorial description 
analogous to Lascoux-Sch\"utzenberger theorem  for the original Kostka polynomials 
([M, III, (6.5)]). 
\par 
We now consider the variety $\SX = G \times V^{r-1}$ for an integer $r \ge 1$, on which $G$ 
acts diagonally, and let $\SX\uni = G\uni \times V^{r-1}$ be the $G$-stable subset of $\SX$.
The variety $\SX$ is called the enhanced variety of level $r$. 
In [S4], the relationship between Kostka functions 
$K^{\pm}_{\Bla, \Bmu}(t)$ indexed by  
$\Bla, \Bmu \in \SP_{n,r}$ and the geometry of $\SX\uni$ was studied. 
In contrast to the case where $r = 1,2$, $\SX\uni$ has infinitely many $G$-orbits if 
$r \ge 3$.  A partition $\SX\uni = \coprod_{\Bla \in \SP_{n,r}}X_{\Bla}$ into
$G$-stable pieces $X_{\Bla}$ was constructed in [S3], 
and some formulas expressing the Kostka functions 
in terms of the intersection cohomology associated to the closure of $X_{\Bla}$ 
were obtained in [S4], 
though it is a partial generalization of the result of Achar-Henderson for the case $r = 2$. 
\par
In this paper, we prove a formula (Theorem 2.6) which is a generalization 
of the formula in [AH, Theorem 4.5] (and also in [FGT (11)]) to arbitrary $r$. 
Combined this formula with the results in [S4], 
we extend some results in [LS] to arbitrary $r$.  In particular, we show 
in the special case where $\Bmu = (-,\dots,-,\mu^{(r)}) \in \SP_{n,r}$ 
(and for arbitrary $\Bla \in \SP_{n,r}$) that $K^{-}_{\Bla,\Bmu}(t)$ has 
a Lasacoux-Sch\"utzenberger type combinatorial description.      

\par

\par\bigskip

\section{Review on Kostka functions}

\para{1.1.}
First we recall basic properties of Hall-Littlewood functions and Kostka polynomials
in the original setting, following [M]. 
Let $\vL = \vL(y) = \bigoplus_{n \ge 0}\vL^n$ be the ring of symmetric functions 
over $\BZ$ with respect to the variables $y = (y_1, y_2, \dots)$, where $\vL^n$ denotes the 
free $\BZ$-module of symmetirc functions of degree $n$. 
We put $\vL_{\BQ} = \BQ\otimes_{\BZ}\vL$, $\vL^n_{\BQ} = \BQ\otimes_{\BZ}\vL^n$. 
Let $s_{\la}$ be the Schur function associated to $\la \in \SP_n$.  Then 
$\{ s_{\la} \mid \la \in \SP_n \}$ gives a $\BZ$-baisis of $\vL^n$. 
Let $p_{\la} \in \vL^n$ be the power sum symmetric function associated to 
$\la \in  \SP_n$,
\begin{equation*}
p_{\la} = \prod_{i = 1}^kp_{\la_i},
\end{equation*}
where $p_m$ denotes the $m$-th power sum symmetric function for each integer $m > 0$. 
Then $\{ p_{\la} \mid \la \in \SP_n \}$ gives a $\BQ$-basis of 
$\vL^n_{\BQ}$. 
For $\la = (1^{m_1}, 2^{m_2}, \dots) \in \SP_n$, define an integer $z_{\la}$ by 
\begin{equation*}
\tag{1.1.1}
z_{\la} = \prod_{i \ge 1}i^{m_i}m_i!.
\end{equation*} 
Following [M, I], we introduce a scalar product on $\vL_{\BQ}$ by 
$\lp p_{\la}, p_{\mu} \rp = \d_{\la\mu}z_{\la}$.  
It is known that $\{s_{\la}\}$ form  an orthonormal basis of $\vL$. 

\par
Let $P_{\la}(y;t)$ be the Hall-Littlewood function associated to a partition $\la$.
Then $\{ P_{\la} \mid \la \in \SP_n \}$ gives a $\BZ[t]$-basis of 
$\vL^n[t] = \BZ[t]\otimes_{\BZ}\vL^n$, where $t$ is an indeterminate.
Kostka polynomials $K_{\la, \mu}(t) \in \BZ[t]$ ($\la, \mu \in \SP_n$) are defined by the formula
\begin{equation*}
\tag{1.1.2}
s_{\la}(y) = \sum_{\mu \in \CP_n}K_{\la,\mu}(t)P_{\mu}(y;t).
\end{equation*}  
\par
Recall the dominance order $\la \ge \mu$ in $\SP_n$, which is defined by the condition 
$\sum_{j= 1}^i\la_j \ge \sum_{j = 1}^i\mu_j$ for 
each $i \ge 1$. 
For each partition $\la = (\la_1, \dots, \la_k)$, we define an integer 
$n(\la)$ by $n(\la) = \sum_{i=1}^k(i-1)\la_i$. 
It is known that $K_{\la,\mu}(t) = 0$ unless $\la \ge \mu$, and that 
$K_{\la,\mu}(t)$ is a monic of degree $n(\mu) - n(\la)$ if $\la \ge \mu$
([M, III, (6.5)]).  
Put $\wt K_{\la,\mu}(t) = t^{n(\mu)}K_{\la, \mu}(t\iv)$.  Then 
$\wt K_{\la, \mu}(t) \in \BZ[t]$, which we call the modified Kostka polynomial.  
\par
For $\la = (\la_1, \dots, \la_k) \in \SP_n$ with $\la_k > 0$, 
we define $z_{\la}(t) \in \BQ(t)$ by 
\begin{equation*}
\tag{1.1.3}
z_{\la}(t) = z_{\la}\prod_{i \ge 1}(1 - t^{\la_i})\iv,
\end{equation*}
where $z_{\la}$ is as in (1.1.1). 
Following [M, III], we introduce a scalar product on $\vL_{\BQ}(t) = \BQ(t)\otimes_{\BZ}\vL$ by 
$\lp p_{\la}, p_{\mu} \rp = z_{\la}(t)\d_{\la,\mu}$.
Then $P_{\la}(y;t)$ form an orthogonal basis of $\vL[t] = \BZ[t]\otimes_{\BZ}\vL$.
In fact, they  are characterized by the following two properties
([M, III, (2.6) and (4.9)]);
\begin{equation*}
\tag{1.1.4}
P_{\la}(y;t) = s_{\la}(y) + \sum_{\mu < \la}w_{\la\mu}(t)s_{\mu}(y)
\end{equation*}
with $w_{\la\mu}(t) \in \BZ[t]$ , and
\begin{equation*}
\tag{1.1.5}
\lp P_{\la}, P_{\mu} \rp = 0 \text{ unless $\la = \mu$. } 
\end{equation*}

\para{1.2.}
We fix a positive integer $r$. 
Let  
$\Xi = \Xi(x) \simeq \vL(x^{(1)})\otimes\cdots\otimes\vL(x^{(r)})$ be 
the ring of symmetric functions over $\BZ$ 
with respect to variables $x = (x^{(1)}, \dots, x^{(r)})$, where 
$x^{(i)} = (x^{(i)}_1, x^{(i)}_2, \dots)$.
We denote it as $\Xi = \bigoplus_{n \ge 0}\Xi^n$, similarly to the case of $\vL$. 
  Let $\SP_{n,r}$ be as in Introduction. 
For $\Bla  \in \SP_{n,r}$, we define a Schur function
$s_{\Bla}(x) \in \Xi^n$ by 
\begin{equation*}
\tag{1.2.1}
s_{\Bla}(x) = s_{\la^{(1)}}(x^{(1)})\cdots s_{\la^{(r)}}(x^{(r)}).
\end{equation*}
Then 
$\{ s_{\Bla} \mid \Bla \in \SP_{n,r} \}$ gives a $\BZ$-basis of $\Xi^n$. 
Let $\z$ be a primitive $r$-th root of unity in $\BC$. 
For an integer $m \ge 1$ and $k$ such that $1 \le k \le r$, put 
\begin{equation*}
p_m^{(k)}(x) = \sum_{j = 1}^{r}\z^{(k-1)(j-1)}p_m(x^{(j)}),
\end{equation*}
where $p_m(x^{(j)})$ denotes the $m$-th power sum symmetric function with respect to 
the variables $x^{(j)}$. 
For $\Bla \in \SP_{n,r}$, we define $p_{\Bla}(x) \in \Xi^n_{\BC} = \Xi^n \otimes_{\BZ}\BC$ by 
\begin{equation*}
\tag{1.2.2}
p_{\Bla}(x) = \prod_{k = 1}^r\prod_{j= 1}^{m_k}p^{(k)}_{\la^{(k)}_j}(x),
\end{equation*}  
where $\la^{(k)} = (\la^{(k)}_1, \dots, \la^{(k)}_{m_k})$ with $\la^{(k)}_{m_k} > 0$. 
Then $\{ p_{\Bla} \mid \Bla \in \SP_{n,r} \}$ gives a $\BC$-basis of 
$\Xi^n_{\BC}$. 
For a partition $\la^{(k)}$ as above, 
we define a function 
$z_{\la^{(k)}}(t) \in \BC(t)$ by  
\begin{equation*}
z_{\la^{(k)}}(t) = \prod_{j = 1}^{m_k}(1 - \z^{k-1}t^{\la_j^{(k)}})\iv.
\end{equation*}
For $\Bla \in \SP_{n,r}$, 
we define an integer $z_{\Bla}$ by $z_{\Bla} = \prod_{k=1}^rr^{m_k}z_{\la^{(k)}}$, 
wher $z_{\la^{(k)}}$ is as in (1.1.1).  
We now define a function $z_{\Bla}(t) \in \BC(t)$ by 
\begin{equation*}
\tag{1.2.3}
z_{\Bla}(t) = z_{\Bla}\prod_{k=1}^rz_{{\la}^{(k)}}(t).
\end{equation*}
Let $\Xi[t] = \BZ[t]\otimes_{\BZ}\Xi$ be the free $\BZ[t]$-module, 
and  $\Xi_{\BC}(t) = \BC(t)\otimes_{\BZ}\Xi$ be the $\BC(t)$-space. 
Then $\{ p_{\Bla}(x) \mid \Bla \in \SP_{n,r} \}$ gives a basis of $\Xi^n_{\BC}(t)$.
We define a sesquilinear form  on $\Xi_{\BC}(t)$ by 
\begin{equation*}
\tag{1.2.4}
\lp p_{\Bla}, p_{\Bmu} \rp = \d_{\Bla,\Bmu}z_{\Bla}(t).
\end{equation*} 
\par
We express an $r$-partition $\Bla = (\la^{(1)}, \dots, \la^{(r)})$ as 
$\la^{(k)} = (\la^{(k)}_1, \dots, \la^{(k)}_m)$ with a common  $m$, 
by allowing zero on parts $\la^{(i)}_j$, and define a composition 
$c(\Bla)$ of $n$ by 
\begin{equation*}
c(\Bla) = (\la^{(1)}_1, \dots, \la^{(r)}_1, \la^{(1)}_2, \dots, \la^{(r)}_2, 
              \dots, \la^{(1)}_m, \dots, \la^{(r)}_m).
\end{equation*} 
We define a partial order $\Bla \ge \Bmu$ on $\SP_{n,r}$ by the condition 
$c(\Bla) \ge c(\Bmu)$, where $\ge $ is the dominance order on the set of 
compositions of $n$ defined in a similar way as in the case of partitions. 
We fix a total order $\Bla \gv \Bmu$ on $\SP_{n,r}$ compatible with the 
partial order $\Bla > \Bmu$. 
\par
The following result was proved in Theorem 4.4 and Proposition 4.8 in [S1], 
combined with [S2, \S 3].

\begin{prop}  %%%%  Prop. 1.3
For each $\Bla \in \SP_{n,r}$, there exist unique functions 
$P^{\pm}_{\Bla}(x;t) \in \Xi^n_{\BQ}(t)$ $($depending on the signs $+$, $-$ $)$
satisfying the following properties.
\begin{enumerate}
\item 
$P^{\pm}_{\Bla}(x;t)$ can be written as 
\begin{equation*}
P^{\pm}_{\Bla}(x;t) = s_{\Bla}(x) + \sum_{\Bmu \lv \Bla}u^{\pm}_{\Bla,\Bmu}(t)s_{\Bmu}(x)
\end{equation*}  
with $u^{\pm}_{\Bla,\Bmu}(t) \in \BQ(t)$. 
\item
$\lp P_{\Bla}^-, P^+_{\Bmu} \rp = 0$ unless $\Bla = \Bmu$.  
\end{enumerate}
\end{prop}

\para{1.4.}
$P^{\pm}_{\Bla}(x;t)$ are called Hall-Littlewood functions associated to 
$\Bla \in \SP_{n,r}$.
By Proposition 1.3, for $\ve \in \{ +,-\}$, $\{ P^{\ve}_{\Bla} \mid \Bla \in \SP_{n,r}\}$ 
gives 
a $\BQ(t)$-basis for $\Xi_{\BQ}(t)$.  
For $\Bla, \Bmu \in \SP_{n,r}$, we define functions $K^{\pm}_{\Bla, \Bmu}(t) \in \BQ(t)$ by 

\begin{equation*}
\tag{1.4.1}
s_{\Bla}(x) = \sum_{\Bmu \in \SP_{n,r}}K^{\pm}_{\Bla,\Bmu}(t)P^{\pm}_{\Bmu}(x;t).
\end{equation*} 

\par
$K^{\pm}_{\Bla, \Bmu}(t)$ are called Kostka functions associated to complex 
reflection groups since they are closely related to the complex reflection group
$S_n\ltimes (\BZ/r\BZ)^n$ (see [S1, Theorem 5,4]).  
For each $\Bla \in \SP_{n,r}$, by putting 
$n(\Bla) = n(\la^{(1)}) + \cdots + n(\la^{(r)})$, we define an $a$-function 
$a(\Bla)$ on $\SP_{n,r}$ by 
\begin{equation*}
\tag{1.4.2}
a(\Bla) = r\cdot n(\Bla) + |\la^{(2)}| + 2|\la^{(3)}| + \cdots + (r-1)|\la^{(r)}|.
\end{equation*}
We define modifed Kostka functions 
$\wt K^{\pm}_{\Bla, \Bmu}(t)$ by 
\begin{equation*}
\tag{1.4.3}
\wt K^{\pm}_{\Bla, \Bmu}(t) = t^{a(\Bmu)}K^{\pm}_{\Bla, \Bmu}(t\iv).
\end{equation*}   
 
\remark{1.5.}
In the case where $r = 1$, $P^{\pm}_{\Bla}(x;t)$ coincides with  the original 
Hall-Littlewood function given in 1.1.  In the case where $r = 2$, 
it is proved by [S2, Prop. 3.3] that $P^-_{\Bla}(x;t) = P^+_{\Bla}(x;t) \in \Xi[t]$, 
hence $K^-_{\Bla,\Bmu}(t) = K^+_{\Bla, \Bmu}(t) \in \BZ[t]$.
Moreover it is shown that $K^{\pm}_{\Bla, \Bmu}(t) \in \BZ[t]$, 
which is a monic of degree $a(\Bmu) - a(\Bla)$.  Thus $\wt K^{\pm}_{\Bla, \Bmu}(t) \in \BZ[t]$.      
As mentioned in Introduction $\wt K^{\pm}_{\Bla, \Bmu}(t)$ has a geometric 
interpretation, which imples that $K^{\pm}_{\Bla, \Bmu}(t)$, and so  
$P^{\pm}_{\Bla}(x;t)$ are independent of the choice of the total order $\lv$ on $\SP_{n,r}$.
In the case where $r \ge 3$, it is not known whether Hall-Littlewood functions do not
depend on the choice of the total order $\lv$, 
whether $K^{\pm}_{\Bla, \Bmu}(t)$ are polynomials in $t$.  

\par\bigskip

\section{Enhanced variety of level $r$ }

\para{2.1.}
Let $V$ be an $n$-dimensional vector space over an algebraic closure $\Bk$ of a finite 
field $\Fq$, 
and $G = GL(V) \simeq GL_n$. 
Let $B = TU$ be a Borel subgroup of $G$, $T$ a maximal torus and $U$ the unipotent radical 
of $B$.  Let $W = N_G(T)/T$ be the Weyl group of $G$, which is isomorphic to the symmetric 
group $S_n$.
By fixing an integer $r \ge 1$, put $\SX = G \times V^{r-1}$
and $\SX\uni = G\uni \times V^{r-1}$, where $G\uni$ is the set of unipotent elements in $G$. 
The variety $\SX$ is called the enhanced variety of level $r$. 
We consider the diagonal action of $G$ on $\SX$.  
Put $\SQ_{n,r} = \{ \Bm = (m_1, \dots, m_r) \in \BZ^r_{\ge 0} \mid \sum m_i = n\}$.
For each $\Bm \in \SQ_{n,r}$, 
we define integers $p_i = p_i(\Bm)$ by $p_i = m_1 + \cdots + m_i$ for $i = 1, \dots, r$.
  Let $(M_i)_{1 \le i \le n}$ be the total flag in $V$ whose stabilizer in $G$ 
coincides with $B$.  We define varieties 

\begin{align*}
\wt\SX_{\Bm} &= \{ (x, \Bv, gB) \in G \times V^{r-1} \times G/B \mid g\iv xg \in B, 
                      g\iv \Bv \in \prod_{i=1}^{r-1}M_{p_i} \}, \\
\SX_{\Bm} &=  \bigcup_{g \in G}g(B \times \prod_{i=1}^{r-1}M_{p_i}),
\end{align*}
and the map $\pi_{\Bm} : \wt \SX_{\Bm} \to \SX_{\Bm}$ by $(x,\Bv, gB) \mapsto (x,\Bv)$.
We also define the varieties

\begin{align*}
\wt\SX_{\Bm, \unip} &= \{ (x, \Bv, gB) \in G\uni \times V^{r-1} \times G/B \mid g\iv xg \in U, 
                      g\iv \Bv \in \prod_{i=1}^{r-1}M_{p_i} \}, \\
\SX_{\Bm} &=  \bigcup_{g \in G}g(U \times \prod_{i=1}^{r-1}M_{p_i}),
\end{align*}
and the map $\pi_{\Bm,1}: \wt\SX_{\Bm,\unip} \to \SX_{\Bm,\unip}$, similarly. 
Note that in the case where $\Bm = (n,0,\dots, 0)$, $\SX_{\Bm}$ (resp. $\SX_{\Bm,\unip}$)
coincides with $\SX$ (resp. $\SX\uni$).  
In that case, we denote $\wt\SX_{\Bm}, \pi_{\Bm}$, etc. 
by $\wt\SX, \pi$, etc.  by omitting the symbol $\Bm$.
(Note: here we follow the notation in [S4], but, in part, it differs from [S3]. 
In [S3], our $\pi_{\Bm}, \pi_{\Bm,1}$ are denoted by $\pi^{(\Bm)}, \pi^{(\Bm)}_1$ 
for the consistency with the exotic case). 

\para{2.2.}
In [S3, 5.3], a partition of $\SX\uni$ into pieces $X_{\Bla}$ is defined

\begin{equation*}
\SX\uni = \coprod_{\Bla \in \SP_{n,r}}X_{\Bla},
\end{equation*}
where $X_{\Bla}$ is a locally closed, smooth irreducible, $G$-stable subvariety
of $\SX\uni$.  If $r = 1$ or 2, $X_{\Bla}$ is a single $G$-orbit. However, if 
$r \ge 3$, $X_{\Bla}$ is in general a union of infinitely many $G$-orbits. 
\par
For $\Bm \in \SQ_{n,r}$, let $W_{\Bm} = S_{m_1} \times \cdots \times S_{m_r}$ 
be the Young subgroup of $W = S_n$.
For $\Bm \in \SQ_{n,r}$, we denote by $\SP(\Bm)$ the set of $\Bla \in \SP_{n,r}$ 
such that $|\la^{(i)}| = m_i$.
The (isomorphism classes of) 
irreducible representations (over $\Ql$) of $W_{\Bm}$ are parametrized by $\SP(\Bm)$.
We denote by $V_{\Bla}$ an irreducible representation of $W_{\Bm}$ corresponding to 
$\Bla$, namley $V_{\Bla} = V_{\la^{(1)}}\otimes\cdots\otimes V_{\la^{(r)}}$,
where $V_{\mu}$ denotes the irreducible representation of $S_n$ corresponding to 
the partition $\mu$ of $n$.  (Here we use the parametrization such that $V_{(n)}$ is the 
trivial representation of $S_n$).
The following results were proved in [S3]. 

\begin{thm}[{[S3, Thm. 4.5]}]  %%%%  Thm 2.3.
Put $d_{\Bm} = \dim \SX_{\Bm}$.  Then $(\pi_{\Bm})_*\Ql[d_{\Bm}]$ is a
semsimple perverse sheaf equipped with the action of $W_{\Bm}$, and is decomposed as

\begin{equation*}
(\pi_{\Bm})_*\Ql[d_{\Bm}] \simeq \bigoplus_{\Bla \in \SP(\Bm)}
         V_{\Bla} \otimes \IC(\SX_{\Bm}, \SL_{\Bla})[d_{\Bm}],
\end{equation*} 
where $\SL_{\Bla}$ is a simple local system on a certain open dense subvariety of $\SX_{\Bm}$.
\end{thm}

\begin{thm}[{[S3, Thm. 8.13, Thm. 7.12]}] %%%% Th. 2.4. 
Put $d'_{\Bm} = \dim \SX_{\Bm,\unip}$.  
\begin{enumerate}
\item
$(\pi_{\Bm,1})_*\Ql[d'_{\Bm}]$ 
is a semisimple perverse sheaf equipped with the action of $W_{\Bm}$, and 
is decomposed as 

\begin{equation*}
(\pi_{\Bm,1})_*\Ql[d'_{\Bm}] \simeq \bigoplus_{\Bla \in \SP(\Bm)}
                  V_{\Bla} \otimes \IC(\ol X_{\Bla}, \Ql)[\dim  X_{\Bla}].
\end{equation*} 
\item
We have 
$\IC(\SX_{\Bm}, \SL_{\la})|_{\SX_{\Bm, \unip}} \simeq 
           \IC(\ol X_{\Bla}, \Ql)[\dim X_{\Bla} - d'_{\Bm}]$. 
\end{enumerate}
\end{thm} 
 
\para{2.5.}
For a partition $\la$, we denote by $\la^t$ the dual partition of $\la$.
For $\Bla = (\la^{(1)}, \dots, \la^{(r)}) \in \SP(\Bm)$, we define 
$\Bla^t \in \SP(\Bm)$ by 
$\Bla^t = ((\la^{(1)})^t, \dots, (\la^{(r)})^t)$.  
Assume that $\Bla  \in \CP(\Bm)$.
We write $(\la^{(i)})^t$  
as ($\mu^{(i)}_1 \le \mu^{(i)}_2 \le \cdots \le \mu^{(i)}_{\ell_i})$,
in the increasing order, where $\ell_i = \la^{(i)}_1$.   
For each $1 \le i \le r, 1\le j < \ell_i$, we define an integer $n(i,j)$ by 
\begin{equation*}
n(i,j) = (|\la^{(1)}| + \cdots + |\la^{(i-1)}|) + \mu_1^{(i)} + \cdots + \mu_j^{(i)}.
\end{equation*} 
Let $Q = Q_{\Bla}$ be the stabilizer of the partial flag $(M_{n(i,j)})$ in $G$, and $U_Q$ 
the unipotent radical of $Q$.
In particular, $Q$ stabilizes the subspaces $M_{p_i}$.  
Let us define a variety $\wt X_{\Bla}$ by 
\begin{equation*}
\begin{split}
\wt X_{\Bla} = \{ (x, \Bv, gQ) \in G\uni \times V^{r-1} \times G/Q
        \mid g\iv xg \in U_Q, g\iv \Bv \in \prod_{i=1}^{r-1}M_{p_i} \}.
\end{split}
\end{equation*}
We define a map $\pi_{\Bla} : \wt X_{\Bla} \to \CX\uni$ by 
$(x,\Bv, gQ) \mapsto (x,\Bv)$. Then $\pi_{\Bla}$ is a proper map. 
Since 
$\wt X_{\Bla} \simeq G\times^{Q}(U_{Q} 
          \times \prod_i M_{p_i})$, 
$\wt X_{\Bla}$ is smooth and irreducible. 
It is known by [S3, Lemma 5.6] that $\dim \wt X_{\Bla} = \dim X_{\la}$ 
and that $\Im \pi_{\Bla}$ coincides with $\ol X_{\Bla}$, the closure of 
$X_{\Bla}$ in $\SX\uni$. 
\par
For $\la, \mu \in \SP_n$, let $K_{\la, \mu} = K_{\la,\mu}(1)$ be the Kostka number. 
We have $K_{\la,\mu} = 0$ unless $\la \ge \mu$. 
For $\Bla = (\la^{(1)}, \dots, \la^{(r)})$, 
$\Bmu = (\mu^{(1)}, \dots, \mu^{(r)}) \in \SP(\Bm)$, we define an integer 
$K_{\Bla, \Bmu}$ by 

\begin{equation*}
K_{\Bla,\Bmu} = K_{\la^{(1)}, \mu^{(1)}}K_{\la^{(2)},\mu^{(2)}}
                   \cdots K_{\la^{(r)}, \mu^{(r)}}.
\end{equation*} 
We define a partial order $\Bla \trreq \Bmu$ on $\SP_{n,r}$ by the condition 
$\la^{(i)} \ge \mu^{(i)}$ for $i = 1, \dots, r$.  
Hence $\Bla \trreq \Bmu$ implies that $\Bla, \Bmu \in \SP(\Bm)$ for a commom $\Bm$.
We have $K_{\Bla, \Bmu} = 0$ 
unless $\Bla \trreq \Bmu$.   
Note that $\Bla \trreq \Bmu$ implies that $\Bmu^t \trreq \Bla^t$.  
We show the following theorem.  In the case where 
$r = 2$, this result was proved by [AH, Thm. 4.5]. 

\begin{thm}  %%%%  Thm. 2.6.
Assume that $\Bla \in \SP_{n,r}$.  Then $(\pi_{\Bla})_*\Ql[\dim X_{\Bla}]$ 
is a semisimple perverse sheaf on $\ol X_{\Bla}$, and is decomposed as 

\begin{equation*}
\tag{2.6.1}
(\pi_{\Bla})_*\Ql[\dim X_{\Bla}] \simeq \bigoplus_{\Bmu \trleq \Bla}
                    \Ql^{K_{\Bmu^t, \Bla^t}}\otimes \IC(\ol X_{\Bmu}, \Ql)[\dim X_{\Bmu}].
\end{equation*}
\end{thm} 

\para{2.7.}
The rest of this section is devoted to the proof of Theorem 2.6.
First we consider the case where $r = 1$.  Actually, the result in this case 
is contained in [AH].  Their proof (for $r = 2$) depends on the result of Spaltenstein [Sp] 
concerning the ``Springer fibre'' $(\pi_{\Bla})\iv(z)$ for $z \in \ol X_{\Bla}$ 
in the case $r = 1$. In the following, we give an alternate proof independent of [Sp] 
for the later use.
Let $Q$ be a parabolic subgroup of $G$ containing $B$, $M$ the Levi 
subgroup of $Q$ containing $T$ and $U_Q$ the unipotent radical of $Q$.
(In this stage, this $Q$ is independent of $Q$ in 2.5.)
Let $W_Q$ be the Weyl subgroup of 
$W$ corresponding to $Q$. 
Let $G\reg$ be the set of regular semisimple elements in $G$, and put $T\reg = G\reg \cap T$.
Consider the map $\psi: \wt G\reg \to G\reg$, where 

\begin{align*}
\wt G\reg = \{ (x, gT) \in G\reg \times G/T \mid g\iv xg \in T\reg \} \\
\end{align*}   
and $\psi : (x, gT) \mapsto x$. 
Then $\psi$ is a finite Galois covering with group $W $. 
We also consider a variety 

\begin{align*}
\wt G\reg^M &= \{ (x, gM) \in G\reg \times G/M \mid g\iv xg \in M\reg \},  
\end{align*} 
where $M\reg = G\reg \cap M$.
The map $\psi$ is decomposed as 

\begin{equation*}
\begin{CD}
\psi : \wt G\reg @>\psi' >>  \wt G\reg^M @> \psi''>>  G\reg, 
\end{CD}
\end{equation*}

\par\noindent
where $\psi': (x, gT) \mapsto (x, gM)$, $\psi'': (x, gM) \mapsto x$.
Here $\psi'$ is a finite Galois covering with group $W_Q$. 
Now $\psi_*\Ql$ is a semisimple local system on $G\reg$ such that 
$\End (\psi_*\Ql) \simeq \Ql[W]$, and 
is decomposed as 

\begin{equation*}
\tag{2.7.1}
\psi_*\Ql \simeq \bigoplus_{\r \in W\wg} \r \otimes \SL_{\r},
\end{equation*}
where $\SL_{\r} = \Hom_W(\r, \psi_*\Ql)$ is a simple local system on $G\reg$. 
We also have
\begin{equation*}
\tag{2.7.2}
\psi'_*\Ql \simeq \bigoplus_{\r' \in W_Q\wg}\r' \otimes \SL'_{\r'},
\end{equation*}
where $\SL'_{\r'}$ is a simple local system on $\wt G\reg^M$. 
Hence 

\begin{equation*}
\tag{2.7.3}
\psi_*\Ql \simeq \psi_*''\psi_*'\Ql \simeq \bigoplus_{\r' \in W_Q\wg}
                           \r'\otimes \psi''_*\SL'_{\r'}.
\end{equation*}
(2.7.3) gives a decompostion of $\psi_*\Ql$ with respect to the action of $W_Q$.
Comparing (2.7.1) and (2.7.3), we have 

\begin{equation*}
\tag{2.7.4}
\psi''_*\SL'_{\r'} \simeq \bigoplus_{\r \in W\wg}\Ql^{(\r: \r')}\otimes \SL_{\r},
\end{equation*}
where $(\r: \r')$ is the multiplicity of $\r'$ in the restricted 
$W_Q$-module $\r$. 
\par
We consider the map $\pi : \wt G \to G$, where 

\begin{equation*}
\wt G = \{ (x, gB) \in G \times G/B \mid g\iv xg \in B \} \simeq G \times^BB ,
\end{equation*}
and $\pi: (x, gB) \mapsto x$.  We also consider 

\begin{align*}
\wt G^Q = \{ (x, gQ) \in G \times G/Q \mid g\iv xg \in Q \} \simeq G \times^QQ. 
\end{align*}
The map $\pi$ is decomposed as 

\begin{equation*}
\begin{CD}
\pi: \wt G @>\pi'>> \wt G^Q @>\pi''>>  G,
\end{CD}
\end{equation*}
where $\pi': (x, gB) \mapsto (x, gQ)$, $\pi'': (x, gQ) \mapsto x$. 
It is well-known ([L1]) that 

\begin{equation*}
\tag{2.7.5}
\pi_*\Ql \simeq \bigoplus_{\r \in W\wg} \r \otimes \IC(G, \SL_{\r}).
\end{equation*}
Let $B_M = B \cap M$ be the Borel subgroup of $M$ containing $T$.
We consider the following commutative diagram

\begin{equation*}
\tag{2.7.6}
\begin{CD}
G \times ^BB @<\wt p<< G \times (Q\times^BB) @>\wt q>> M \times^{B_M}B_M \\
    @V\pi'VV                   @VVr V                @VV\pi^M V   \\
G \times^QQ @<p<<   G \times Q  @>q>> M ,  
\end{CD}
\end{equation*} 
where under the identification $G \times^BB \simeq G \times^Q(Q \times^BB)$, 
the maps $p,\wt p$ are defined by the quotient by $Q$. The map $q$ is 
a projection to the $M$-factor of $Q$, and $\wt q$ is the map induced 
from the projection $Q \times B \to M \times B_M$. 
$\pi^M$ is defined similarly to $\pi$ replacing $G$ by $M$.  The map $r$ 
is defined by $(g, h*x) \mapsto (g, hxh\iv)$.
(We use the notation $h*x \in Q\times^BB$ to denote the $B$-orbit in $Q \times B$
containing $(h,x)$.)  
Here all the squares are cartesian squares.  Moreover, 
\par\medskip
(a) $p$ is a principal $Q$-bundle.
\par
(b) $q$ is a locally trivial fibration with fibre isomorphic to $G \times U_Q$.
\par\medskip\noindent
Thus as in [S4, (1.5.2)], for any $M$-equivariant simple pervere sheaf $A_1$ 
on $M$, there exists a unique (up to isomorphism) simple perverse sheaf $A_2$ 
on $\wt G^Q$ such that $p^*A_2[a] \simeq q^*A_1[b]$, where 
$a = \dim Q$ and $b = \dim G + \dim U_Q$.   
\par
By using the cartesian squares in (2.7.6), and by (2.7.2), 
we see that $\pi'_*\Ql \simeq \IC(\wt G^Q, \psi'_*\Ql)$, and $\pi'_*\Ql$ is decomposed as  
\begin{equation*}
\tag{2.7.7}
\pi'_*\Ql \simeq \bigoplus_{\r' \in W_Q\wg}\r' \otimes \IC(\wt G^Q, \SL'_{\r'}).
\end{equation*} 

By comparing (2.7.4) and (2.7.7), we have

\begin{equation*}
\tag{2.7.8}
\pi''_*\IC(\wt G^Q, \SL'_{\r'}) \simeq \bigoplus_{\r \in W\wg}\Ql^{(\r:\r')}\otimes 
      \IC(G,\SL_{\r}).
\end{equation*}
Note that if $\r = V_{\la}$ for $\la \in \SP_n$, we have 
\begin{equation*}
\tag{2.7.9}
\IC(G, \SL_{\r})|_{G\uni} \simeq \IC(\ol\SO_{\la}, \Ql)[\dim \SO_{\la} - 2\nu_G] 
\end{equation*}
by [BM], where $\nu_G = \dim U$.  
Hence by restricting on $G\uni$, we have 

\begin{equation*}
\tag{2.7.10}
\pi''_*\IC(\wt G^Q, \SL'_{\r'})[2\nu_G]|_{G\uni} 
      \simeq \bigoplus_{\la \in \SP_n}\Ql^{(V_{\la} : \r')}\otimes 
      \IC(\ol\SO_{\la}, \Ql)[\dim \SO_{\la}].
\end{equation*}

\para{2.8.}
Now assume that 
$W_Q \simeq S_{\mu}$ for a partition $\mu$, where we put 
$S_{\mu} = S_{\mu_1} \times \cdots \times S_{\mu_k}$ if  
$\mu = (\mu_1, \dots, \mu_k) \in \SP_n$.  
Take $\r' = \ve$  the sign representation of $W_Q$. 
We have 
\begin{equation*}
\tag{2.8.1}
(V_{\la} : \ve) = (V_{\la^t} : 1_{W_Q}) = K_{\la^t,\mu},
\end{equation*}
where $1_{W_Q}$ is the trivial representation of $W_Q$.  
\par
The restriction of the diagram  (2.7.6) to the ``unipotent parts'' makes sense, and 
we have the commutative diagram 

\begin{equation*}
\tag{2.8.2}
\begin{CD}
G \times^BU @<<<  G \times^Q(Q \times^BU) @>>>  M \times^{B_M}U_M   \\
    @VVV                  @VVV                     @VVV   \\   
G \times^QQ\uni @<p_1<<  G \times Q\uni @>q_1>> M\uni,
\end{CD}
\end{equation*}
where $U_M$ is the unipotent radical of $B_M$, and $Q\uni, M\uni$ are the set of 
unipotent elements in $Q, M$, respectively.
$p_1, q_1$ have similar properties as (a), (b) in 2.7.
We consider $\IC(M, \SL^M_{\ve})$ on $M$, where $\SL^M_{\ve}$ is the simple local 
system on $M\reg$ corresponding to 
$\ve \in W\wg_Q$.  Then  by (2.7.6), we see that 

\begin{equation*}
p^*\IC(\wt G^Q, \SL'_{\ve}) \simeq q^*\IC(M, \SL^M_{\ve}).
\end{equation*}  
By applying (2.7.9) to $M$, 
$\IC(M, \SL^M_{\ve})|_{M\uni} \simeq \IC(\ol\SO'_{\ve}, \Ql)[\dim \SO'_{\ve} - 2\nu_M]$, 
 where $\SO'_{\ve}$ is the orbit in $M\uni$ corresponding to 
$\ve$ under the Springer correspondence, and $\nu_M$ is defined similarly to $\nu_G$.
It is known that $\SO'_{\ve}$ is the orbit $\{ e \} \subset M\uni$, where $e$ is the identity element 
in $M$.  Hence $\IC(M, \SL^M_{\ve})|_{M\uni}$ coincides with $\Ql[-2\nu_M]$ supported on 
$\{e\}$.   
It follows, by (2.8.2) 

\par\medskip\noindent
(2.8.3) \ The restriction of $\IC(\wt G^Q, \SL'_{\ve})$ on $G\times^QQ\uni$ 
coincides with $i_*\Ql[-2\nu_M]$, where $i: G \times^QU_Q \hra G \times^QQ\uni$ is 
the closed embedding. 
\par\medskip

We deifne a map $\pi_Q : G\times^QU_Q \to G\uni$ by $g*x \mapsto gxg\iv$.  
Put $\wt G^Q_1 = G\times^QU_Q$.  
%%%%
%%%%
\begin{prop}  %%%%  Prop. 2.9.
Under the notation as above, 
\begin{enumerate}
\item
$\pi''_*\IC(\wt G^Q, \SL'_{\ve})[2\nu_G]|_{G\uni} \simeq (\pi_Q)_*\Ql[\dim \wt G^Q_1]$.  

\item
We have
\begin{equation*}
(\pi_Q)_*\Ql[\dim \wt G^Q_1] \simeq \bigoplus_{\substack{\mu \in \SP_n \\ \mu \le {}^t\la}}
                    \Ql^{K_{{}^t\la, \mu}}\otimes \IC(\ol\SO_{\la}, \Ql)
       [\dim \SO_{\la}].
\end{equation*}
\end{enumerate}
\end{prop} 

\begin{proof}
Note that $2\nu_G - 2\nu_M = 2\dim U_Q = \dim \wt G^Q_1$.  
Thus by (2.8.3), 
\begin{equation*}
\tag{2.9.1}
\IC(\wt G^Q, \SL'_{\ve})[2\nu_G]|_{G \times^QQ\uni} \simeq i_*\Ql[\dim \wt G^Q_1]. 
\end{equation*}
By applying the base change theorem to the cartesian square

\begin{equation*}
\begin{CD}
G \times^QQ\uni @>>>  G \times^QQ \\
@V\pi_1''VV                    @VV\pi''V   \\
G\uni        @>>>   G, 
\end{CD}
\end{equation*}
we obtain (i) from (2.9.1) since $\pi_Q = \pi''_1\circ i$. 
Then (ii) follows from  (i) 
by using (2.7.10) and (2.8.1). 
\end{proof}

\para{2.10.}
Returning to the setting in 2.5, we consider the case where $r$ is arbitrary. 
We fix $\Bm \in \SQ_{n,r}$, and let $P = P_{\Bm}$ be the parabolic subgroup of 
$G$ containing $B$ which is the stabilizer of the partial flag $(M_{p_i})_{1 \le i \le r}$. 
Let $L$ be the Levi subgroup of $P$ containing $T$, and $B_L = B \cap L$ the Borel 
subgroup of $L$ containing $T$. Let $U_L$ be the unipotent radical of $B_L$.   
Put $\ol M_{p_i} = M_{p_i}/M_{p_{i-1}}$ for each $i$, under the convention $M_{p_0} = 0$.
Then $L$ acts naturally on $\ol M_{p_i}$, and by applying the definition of 
$\pi_{\Bm,1} : \wt\SX_{\Bm,\unip} \to \SX_{\Bm, \unip}$ to $L$, we can define 

\begin{align*}
\wt\SX^L_{\Bm, \unip} &\simeq L \times^{B_L}(U_L \times \prod_{i=1}^{r-1}\ol M_{p_i}), \\
\SX^L_{\Bm,\unip} &= \bigcup_{g \in L}g(U_L \times \prod_{i = 1}^{r-1}\ol M_{p_i}) 
= L\uni \times \prod_{i=1}^{r-1} \ol M_{p_i}
\end{align*}
and the map $\pi^L_{\Bm,1} : \wt\SX^L_{\Bm, \unip} \to \SX^L_{\Bm, \unip}$ similarly. 

Let $Q = Q_{\Bla}$ be as in 2.5 for $\Bla \in \SP(\Bm)$. 
Thus we have $B \subset Q \subset P$, and
$Q_L = Q \cap L$ is a parabolic subgroup of 
$L$ containing $B_L$. 
We consider the following commutative diagram
\begin{equation*}
\tag{2.10.1}
\begin{CD}
\wt\SX_{\Bm,\unip} @<\wt p_1<<  G \times \wt\SX^P_{\Bm,\unip} @>\wt q_1>>  \wt\SX^L_{\Bm,\unip}  \\
     @V\a'_1VV                 @VV r'_1 V                       @VV\b'_1 V    \\
\wh \SX^Q_{\Bm,\unip}  @<\wh p_1 <<  G \times \wt\SX_{\Bm,\unip}^{P,Q}  @>\wh q_1>>  
                                  \wt\SX^{L, Q_L}_{\Bm,\unip}  \\
        @V\a''_1VV                      @VVr_1''V                        @VV\b''_1V  \\
\wh \SX^P_{\Bm, \unip}    @<p_1<<   G \times \SX_{\Bm, \unip}^P @>q_1>>  \SX^L_{\Bm,\unip}  \\
      @V\pi''_1 VV                                                       \\
   \SX_{\Bm,\unip}, 
\end{CD}
\end{equation*}
where, by putting $P\uni = L\uni U_P$ (the set of unipotent elements in $P$),  
\begin{align*}
\SX_{\Bm,\unip}^P &= \bigcup_{g \in P}g(U \times \prod_i M_{p_i}) = P\uni \times \prod_i M_{p_i}, \\ 
\wh \SX^P_{\Bm,\unip} &= G \times^P\SX_{\Bm,\unip}^P = G \times^P(P\uni \times \prod_iM_{p_i}), \\ 
\wt \SX^P_{\Bm,\unip} &= P \times^{B}(U \times \prod_iM_{p_i}), \\
\wh \SX^Q_{\Bm,\unip} &= G \times^Q(Q\uni \times \prod_i M_{p_i}), \\
\wt\SX^{P,Q}_{\Bm, \unip} &= P \times^Q(Q\uni \times \prod_i M_{p_i}).  
\end{align*}
$\wt\SX^{L,Q_L}_{\Bm \unip}$ is a similar variety as $\wh\SX^P_{\Bm, \unip}$ 
defined with respecto to $(L, Q_L)$, namely, 

\begin{equation*}
\wt\SX^{L, Q_L}_{\Bm, \unip} =  L \times^{Q_L}((Q_L)\uni \times \prod_i\ol M_{p_i}). 
\end{equation*}
The maps are defined as follows; 
under the identification $\wt\SX_{\Bm, \unip} \simeq G \times^B(U \times \prod_iM_{p_i})$, 
$\a'_1, \a''_1$ are the natural maps induced from the inclusions 
$G \times (U \times \prod M_{p_i}) \to G \times (Q\uni \times \prod M_{p_i}) 
      \to G \times (P\uni \times \prod M_{p_i})$. 
$\pi_1'': g*(x,\Bv) \mapsto (gxg\iv, g\Bv)$. 
$q_1$ is defined by $(g,x,\Bv) \mapsto (\ol x, \ol \Bv)$, where 
$x \to \ol x$, $\Bv \mapsto \ol \Bv$ are natural maps 
$P \to L, \prod_iM_{p_i} \to \prod_i\ol M_{p_i}$. 
$\wt q_1$ is the composite of the projection 
$G \times \wt\SX^P_{\Bm,\unip} \to \wt\SX^P_{\Bm, \unip}$
and the map $\wt\SX^P_{\Bm, \unip} \to \wt\SX^L_{\Bm,\unip}$ induced from 
the projection $P \times (U \times \prod M_{p_i}) \to L \times (U_L \times \prod \ol M_{p_i})$.
$\wh q_1$ is defined similarly by using the map 
$\wt\SX^{P,Q}_{\Bm,\unip} \to \wh\SX^{L,Q_L}_{\Bm,\unip}$ 
induced from the projection $P \times (Q\uni \times \prod M_{p_i}) 
\to L \times ((Q_L)\uni \times \prod \ol M_{p_i})$.
$p_1$ is the quotient by $P$.  $\wt p_1$ and $\wh p_1$ are also quotient by $P$
under the identifications $\wt\SX_{\Bm,\unip} \simeq G \times^P \wt\SX^P_{\Bm,\unip}$,
$\wh\SX^Q_{\Bm\unip} \simeq G \times^P\wt\SX^{P,Q}_{\Bm, \unip}$.  
$\b_1'$ is defined similarly to $\a_1'$ and $\b_1''$ is defined similarly to $\pi_1''$. 
$r'_1$ is the natural map induced from the injection 
$P \times (U \times \prod M_{p_i}) \to P \times (Q\uni \times \prod M_{p_i})$, 
and $r_1''$ is the natural map induced from the map 
$P \times^Q(Q\uni \times \prod M_{p_i}) \to P\uni \times \prod M_{p_i}$, 
$g*(x,\Bv) \mapsto (gxg\iv, g\Bv)$.
\par
Put $\pi'_1 = \a_1''\circ \a_1': \wt\SX_{\Bm, \unip} \to \wh\SX^P_{\Bm, \unip}$. 
We have $\b_1''\circ \b_1' = \pi^L_{\Bm,1}$, and the diagram (2.10.1) is the refinement of 
the diagram (6.3.2) in [S4] (see also the diagram (1.5.1) in [S4]). 
In particular, the map $p_1$ is a principal $P$-bundle, and the map $q_1$ is a locally 
trivial fibration with fibre isomorphic to $G \times U_P \times \prod_{i=1}^{r-2}M_{p_i}$. 
Moreover, all the squares appearing in (2.10.1) are caetesian squares. 
Hence the diagram (2.10.1) satisfies similar properties as in the diagram (2.8.2).
\par
Note that $L \simeq G_1 \times \cdots \times G_r$ with $G_i = GL(\ol M_{p_i})$. 
Then $Q_L$ can be written as $Q_L \simeq Q_1 \times \cdots \times Q_r$, where 
$Q_i$ is a parabloic subgroup of $G_i$. 
We have 
\begin{align*}
\wt\SX^L_{\Bm, \unip} &\simeq \prod_{i=1}^r(\wt G_i)\uni \times V,  \\
\wh \SX^{L,Q_L}_{\Bm, \unip} &\simeq \prod_{i=1}^r (\wt G_i^{Q_i})\uni \times V, \\
\SX^L_{\Bm, \unip} &\simeq \prod_{i=1}^r (G_i)\uni \times V,
\end{align*}
where $(\wt G_i)\uni, (\wt G_i^{Q_i})\uni$, etc. denote the unipotent parts of 
$\wt G_i, \wt G_i^{Q_i}$, etc. as in (2.8.2). The maps $\b_1', \b_1''$ are induced from 
the maps $(\wt G_i)\uni \to (\wt G_i^{Q_i})$, $(\wt G_i^{Q_i})\uni \to (G_i)\uni$, and those
maps coincide with the maps $\pi', \pi''$ in 2.7 defined with respect to $G_i$.   
Note that $W_{Q_i} \simeq S_{(\la^{(i)})^t}$ for each $i$ by 
the construction of $Q = Q_{\Bla}$ in 2.5.
Put 
\begin{equation*}
\wh \SX^Q_1 = G \times^Q(U_Q \times \prod M_{p_i}),  \quad 
\wt\SX^{L,Q_L}_1 = L \times^{Q_L}(U_{Q_L} \times \prod \ol M_{p_i}), 
\end{equation*}  
and let $i_Q : \wh\SX_1^Q \hra \wh \SX^Q_{\Bm, \unip}, 
         i_{Q_L} : \wt \SX^{L, Q_L}_1 \hra  \wt \SX^{L,Q_L}_{\Bm \unip}$ be 
the closed embeddings. 
Let $\pi^L_{Q_L} : \wt \SX^{L,Q_L}_1 \to \SX^L_{\Bm, \unip}$ be 
the restriction of $\b_1''$.
Let $\SO^L_{\Bmu} \simeq \SO'_{\mu^{(1)}} \times \cdots \times \SO'_{\mu^{(r)}}$ be 
the $L$-orbit in $\SX^L_{\Bm, \unip}$, 
where $\SO'_{\mu^{(i)}}$ is the $G_i$-oribt in $(G_i)\uni \times \ol M_{p_i}$ of type 
$(\mu^{(i)}, \emptyset)$.  Note that if we denote by $\SO_{\mu^{(i)}}$ the $G_i$-orbit
in $(G_i)\uni$ of type $\mu^{(i)}$,  we have 
$\IC(\ol\SO'_{\mu^{(i)}}, \Ql) \simeq \IC(\ol\SO_{\mu^{(i)}}, \Ql) \boxtimes \Ql$ 
(the latter term $\Ql$ denotes the constatn sheaf on $\ol M_{p_i}$).
Hence the decompostion of $\pi^L_{Q_L}$ into simple components is described by 
considering the factors
$\IC(\ol \SO_{\mu^{(i)}}, \Ql)$.
In particular, by Proposition 2.9, we have 

\begin{equation*}
\tag{2.10.2}
(\pi^L_{Q_L})_*\Ql[\dim  \wt\SX_1^{L,Q_L}] \simeq 
          \bigoplus_{\Bmu \trleq \Bla}
        \Ql^{K_{\Bmu^t, \Bla^t}}\otimes \IC(\ol \SO^L_{\Bmu}, \Ql)[\dim \SO^L_{\Bmu}]. 
\end{equation*} 
\par
By using the diagram (2.10.1), we see that 
\begin{equation*}
\wh q_1^*(i_{Q_L})_*\Ql[\dim \wt\SX^{L, Q_L}_1] \simeq \wh p_1^*(i_Q)_*\Ql[\dim \wt X_{\Bla}]. 
\end{equation*}
It follows, again by using the diagram (2.10.1), we have

\begin{equation*}
\tag{2.10.3}
(\a_1'')_*(i_Q)_*\Ql[\dim \wt X_{\Bla}] \simeq \bigoplus_{\Bmu \trleq \Bla}
            \Ql^{K_{\Bmu^t, \Bla^t}}\otimes B_{\Bmu},
\end{equation*}
where $B_{\Bmu}$ is the simple perverse sheaf on $\wh\SX^P_{\Bm, \unip}$ 
characterized by the property that 
\begin{equation*}
 p_1^*B_{\Bmu}[a'] \simeq  q_1^*\IC(\ol\SO^L_{\Bmu}, \Ql)[b' + \dim \SO^L_{\Bmu}]
\end{equation*}
with $a' = \dim P$, $b' = \dim G + \dim U_P + \dim \prod_{i=1}^{r-2}M_{p_i}$.
\par
On the other hand, by Proposition 1.6 in [S4], we have 
\begin{equation*}
\pi''_*A_{\Bmu} \simeq \IC(\SX_{\Bm}, \SL_{\Bmu})[d_{\Bm}],
\end{equation*}
where $\pi'': \wh\SX^P_{\Bm} = G \times^P(P \times \prod_iM_{p_i}) 
      \to \SX_m $
is an analogous map to $\pi''_1$, and $A_{\Bmu}$ is a simple perverse sheaf on $\wh \SX^P_{\Bm}$
such that the restriction of $A_{\Bmu}$ on $\wh\SX^P_{\Bm, \unip}$ coincides with $B_{\Bmu}$, 
up to shift.
Thus by Theorem 2.4 (ii), we have

\begin{equation*}
\tag{2.10.4}
(\pi''_1)_*B_{\Bmu} \simeq \IC(\ol X_{\Bmu}, \Ql)[\dim X_{\Bmu}]. 
\end{equation*} 
Since $\pi_{\Bla} = \pi_1''\circ \a_1''\circ i_Q$,  by applying $(\pi''_1)_*$ on both sides of 
(2.10.3), we obtain the formula (2.6.1).  This completes the proof of Theorem 2.6.

\par\bigskip
\section{$G^F$-invariant functions on the enhanced variety \\ 
and Kostka functions}

\para{3.1.}
We now assume that $G$ and $V$ are defined over $\Fq$, and let 
$F: G \to G, F: V \to V$ be the corresponding Frobenius maps. 
Assume that $B$ and $T$ are $F$-stable.  
Then $X_{\Bla}$ and $\wt X_{\Bla}$ have natrual $\Fq$-structures, and the map 
$\pi_{\Bla}: \wt X_{\Bla} \to \ol X_{\Bla}$ is $F$-equivariant. 
Thus  one can define a canonical isomorphsim $\vf : F^*K_{\Bla} \isom K_{\Bla}$ 
for $K_{\Bla} = (\pi_{\Bla})_*\Ql$. 
By using the decomposition in Theorem 2.6, 
$\vf$ can be written as $\vf = \sum_{\Bmu}\s_{\Bmu} \otimes \vf_{\Bmu}$, 
where $\s_{\Bmu}$ is the identity map on $\Ql^{K_{\Bmu^t, \Bla^t}}$ and 
$\vf_{\Bmu} : F^*L_{\Bmu} \isom  L_{\Bmu}$ is the isomorphism induced from $\vf$ 
for $L_{\Bmu} = \IC(\ol X_{\Bmu}, \Ql)$.  
(Note that $\dim X_{\Bla} - \dim X_{\Bmu}$ is even if $\Bmu \trleq \Bla$ by 
[S4, Prop. 4.3], so the degree shift is negligible). 
We also consider the natural isomorphism $\f_{\Bmu} : F^*L_{\Bmu} \isom L_{\Bmu}$
induced from the $\Fq$-strucutre of $X_{\Bmu}$.  
By using a similar argument as in [S4, (6.1.1)], we see that
\begin{equation*}
\tag{3.1.1}
\vf_{\Bmu} = q^{d_{\Bmu}}\f_{\Bmu},
\end{equation*}
where $d_{\Bmu} = n(\Bmu)$.  
We consider the characteristic function $\x_{L_{\Bmu}}$ of $L_{\Bmu}$
with respect to $\f_{\Bmu}$, which is a $G^F$-invariant function on $\ol X_{\Bmu}^F$. 

\para{3.2.}
Take $\Bmu, \Bnu \in \SP_{n,r}$, and assume that $\Bnu \in \SP(\Bm)$.  
For each $z = (x, \Bv) \in X_{\Bmu}$ with $\Bv = (v_1, \dots, v_{r-1})$, we define a
variety $\SG_{\Bnu,z}$ by 

\begin{equation*}
\tag{3.2.1}
\begin{split}
\SG_{\Bnu,z} = \{ (W _{p_i}) &\text{ : $x$-stable flag } \mid v_i \in W_{p_i} 
         \ (1 \le i \le r-1), \\
             &x|_{W_{p_i}/W_{p_{i-1}}} 
                             \text{: type $\nu^{(i)}$ } \ (1 \le i \le r)  \}.
\end{split}
\end{equation*}
If $z \in X_{\Bmu}^F$, the variety $\SG_{\Bnu,z}$ is defined over $\Fq$. 
Put $g_{\Bnu,z}(q) = |\SG_{\Bnu,z}^F|$. 
Let $\wt K_{\la,\mu}(t)$ be the modified Kostka polynomial indexed by partitions 
$\la, \mu$. 
The following result is a generalization of Proposiition 5.8 in [AH]. 
%%%%
%%%%
\begin{prop}  %%%%  Prop. 3.3
Assume that $\Bla, \Bmu \in \SP_{n,r}$.  For each $z \in X_{\Bmu}^F$, we have 

\begin{equation*}
\x_{L_{\Bla}}(z) = q^{-n(\Bla)}\sum_{\Bnu \trleq \Bla}g_{\Bnu, z}(q)
                     \wt K_{\la^{(1)},\nu^{(1)}}(q)\cdots \wt K_{\la^{(r)},\nu^{(r)}}(q).
\end{equation*}
\end{prop}

\begin{proof}
Let $\x_{K_{\Bla}, \vf}$ be the characteristic function of $K_{\Bla}$ with respect to 
$\vf$.  
By Theorem 2.6 together with (3.1.1), we have

\begin{equation*}
\tag{3.3.1}
\x_{K_{\Bla}, \vf} = \sum_{\Bxi \trleq \Bla}K_{\Bxi^t, \Bla^t}q^{n(\Bxi)}\x_{L_{\Bxi}}.
\end{equation*}  
On the other hand, by the Grothendieck's fixed point formula, we have 
$\x_{K_{\Bla}, \vf}(z) = |\pi_{\Bla}\iv(z)^F|$ for $z \in \ol X_{\Bla}^F$. 
Then if $z = (x, \Bv) \in X_{\Bmu}^F$, 

\begin{equation*}
\tag{3.3.2}
|\pi_{\Bla}\iv(z)^F| = \sum_{\Bnu \in \SP_{n,r}}|\SG_{\Bnu,z}^F|\prod_i|\pi_{\la^{(i)}}\iv(x_i)^F|, 
\end{equation*}
where $\pi_{\la^{(i)}} : \wt\SO_{\la^{(i)}} \to \ol\SO_{\la^{(i)}}$ is a similar map 
as $\pi_{\Bla}$ applied to the case $r = 1$, by replacing $G$ by $G_i = GL(\ol M_{p_i})$, 
and $x_i = x|_{\ol M_{p_i}}$ has Jordan type $\nu^{(i)}$.   
It is known by [L1] that  
$q^{n(\xi^{(i)})}\x_{L_{\xi^{(i)}}}(x_i) = \wt K_{\xi^{(i)}, \nu^{(i)}}(q)$ 
for a partition $\xi^{(i)}$ of $m_i$.
It follows, by applying (3.3.1) to the case where $r = 1$, 
and by the Grothendieck's fixed point formula, 
we have
\begin{equation*}
|\pi_{\la^{(i)}}\iv(x_i)^F| = \sum_{\xi^{(i)} \le \la^{(i)}}
          K_{\xi^{(i)t}, \la^{(i)t}}\wt K_{\xi^{(i)}, \nu^{(i)}}(q). 
\end{equation*}
Then (3.3.2) implies that 

\begin{equation*}
\tag{3.3.3}
\x_{K_{\Bla}, \vf} = 
|\pi_{\Bla}\iv(z)^F| = \sum_{\Bnu \in \SP_{n,r}}g_{\Bnu,z}(q)\sum_{\Bxi \trleq \Bla}
                        K_{\Bxi^t, \Bla^t}\wt K_{\xi^{(1)}, \nu^{(1)}}(q)\cdots
                            \wt K_{\xi^{(r)}, \nu^{(r)}}(q).
\end{equation*}
Since $(K_{\Bxi^t, \Bla^t})_{\Bla, \Bxi}$ is a unitriangular matrix with respect 
to the partial order $\Bxi \trleq \Bla$, by comparing (3.3.1) and (3.3.3), 
we obtain the required formula. 
\end{proof}

\remark{3.4.}
In general, $X_{\Bmu}$ consists of infinitely many $G$-orbits. Hence the value 
$g_{\Bnu, z}(q)$ may depend on the choice of $z \in X_{\Bmu}^F$.  However, if 
$X_{\Bmu}$ is a single $G$-orbit, then $X_{\Bmu}^F$ is also a single $G^F$-orbit, and 
$g_{\Bnu,z}(q)$ is constant for $z \in X_{\Bmu}^F$, in which case, we denote 
$g_{\Bnu,z}(q)$ by $g_{\Bnu}^{\Bmu}(q)$.  In what follows, we show in some special 
cases that there exists a polynomial $g_{\Bnu}^{\Bmu}(t) \in \BZ[t]$ such that 
$g_{\Bnu}^{\Bmu}(q)$ coincides with the value at $t = q$ of $g_{\Bnu}^{\Bmu}(t)$. 

\para{3.5.}
We consider the special case where $\Bmu \in \SP(\Bm')$ is such that 
$m_i' = 0$ for $i = 1, \dots, r-2$.  In this case, $X_{\Bmu}$ consists of a single 
$G$-orbit. 
In particular, for $\Bla \in \SP_{n,r}$, 
$\dim \SH^i_z\IC(\ol X_{\Bla}, \Ql)$ does not depend on the chocie 
of $z \in X_{\Bmu}$.  
We define a polynomial $\IC^-_{\Bla, \Bmu}(t) \in \BZ[t]$ by 

\begin{equation*}
\IC^-_{\Bla,\Bmu}(t) = \sum_{i \ge 0}\dim \SH^{2i}_z\IC(\ol X_{\Bla}, \Ql)t^i.
\end{equation*}

The following result was proved in [S4].

\begin{prop}[{[S4, Prop. 6.8]}]   %%%%  Prop. 3.6
Let $\Bla, \Bmu \in \SP_{n,r}$, and assume that $\Bmu$ is as in 3.5.
\begin{enumerate}
\item
Assume that $z \in X^F_{\Bmu}$.  Then $\SH^i_z\IC(\ol X_{\Bla}, \Ql) = 0$ if $i$ is odd, 
and the eigenvalues of $\f_{\Bla}$ on $\SH^{2i}_z\IC(\ol X_{\Bla}, \Ql)$ are $q^i$.
In particular, $\x_{L_{\Bla}}(z) = \IC^-_{\Bla, \Bmu}(q)$. 
\item 
$\wt K^-_{\Bla, \Bmu}(t) = t^{a(\Bla)}\IC^-_{\Bla, \Bmu}(t^r)$. 
\end{enumerate}
\end{prop}

As a corollary, we have the following result, which is a 
generalization of [AH, Prop. 5.8] (see also [LS, Prop. 3.2]).  

\begin{cor}  %%%%%  Cor. 3.7.
Assume that $\Bmu$ is as in 3.5.  
\begin{enumerate}
\item
There exists a polynomial $g_{\Bnu}^{\Bmu}(t) \in \BZ[t]$ 
such that $g_{\Bnu}^{\Bmu}(q)$ coincides with the value at 
$t = q$ of $g^{\Bmu}_{\Bnu}(t)$.  
\item
We have
\begin{equation*}
\tag{3.7.1}
\wt K^-_{\Bla, \Bmu}(t) = t^{a(\Bla)-rn(\Bla)}
    \sum_{\Bnu \trleq \Bla}g^{\Bmu}_{\Bnu}(t^r)
         \wt K_{\la^{(1)}, \nu^{(1)}}(t^r)\cdots \wt K_{\Bla^{(r)}, \Bnu^{(r)}}(t^r).
\end{equation*}
\end{enumerate}
\end{cor}

\begin{proof}
By Proposition 3.6 (i) and Proposition 3.3, we have

\begin{equation*}
\tag{3.7.2}
\IC^-_{\Bla, \Bmu}(q) = q^{-n(\Bla)}\sum_{\Bnu \trleq \Bla}g_{\Bnu}^{\Bmu}(q)
                         \wt K_{\la^{(1)},\nu^{(1)}}(q)\cdots \wt K_{\la^{(r)}, \nu^{(r)}}(q)
\end{equation*}
By fixing $\Bmu$, we consider two sets of functions 
$\{ \IC^-_{\Bla \Bmu}(q) \mid \Bla \in \SP_{n,r} \}$ and 
$\{ g^{\Bmu}_{\Bnu}(q) \mid \Bnu \in \SP_{n,r} \}$.   
If we notice that $\wt K_{\la^{(1)}, \nu^{(1)}}(q)\cdots \wt K_{\la^{(r)}, \nu^{(r)}}(q) 
        = q^{n(\Bla)}$ for $\Bnu = \Bla$, (3.7.2) shows that the transition matrix 
between those two sets is unitriangular. 
Hence $g^{\Bmu}_{\Bnu}(q)$ is determined from $\IC^-_{\Bla, \Bmu}(q)$, and a similar 
formula makes sense if we replace $q$ by $t$. This implies (i).  
(ii) now follows from (3.7.2) by replacing $q$ by $t$. 
\end{proof}

\para{3.8.}
In what follows, we assume that $\Bmu$ is of the form 
$\Bmu = (-, \dots, -, \xi)$ with $\xi \in \SP_n$. 
In this case, $g^{\Bmu}_{\Bnu}(t)$ coincides with the polynomial 
$g^{\xi}_{\nu^{(1)}, \dots, \nu^{(r)}}(t)$ 
obtained from  
$G^{\xi}_{\nu^{(1)}, \dots, \nu^{(r)}}(\Fo)$ discussed in [M, II, 2].  
On the other hand, we define a polynomial $f^{\xi}_{\nu^{(1)}, \dots, \nu^{(r)}}(t)$ by 
\begin{equation*}
\tag{3.8.1}
P_{\nu^{(1)}}(y;t)\cdots P_{\nu^{(r)}}(y;t) = 
       \sum_{\xi \in \SP_n}f^{\xi}_{\nu^{(1)}, \dots, \nu^{(r)}}(t)P_{\xi}(y;t).
\end{equation*}
In the case where $r = 2$, $g^{\xi}_{\nu^{(1)}, \nu^{(2)}}(t)$ coincides with the Hall polynomial, 
and a simple formula relating it with $f^{\xi}_{\nu^{(1)}, \nu^{(2)}}(t)$ is konwn 
([M, III (3.6)]).  In the general case, we also have a formula 

\begin{equation*}
\tag{3.8.2}
g^{\xi}_{\nu^{(1)}, \dots, \nu^{(r)}}(t) = t^{n(\xi)- n(\Bnu)}
                      f^{\xi}_{\nu^{(1)}, \dots, \nu^{(r)}}(t\iv).
\end{equation*} 
The proof is easily reduced to [M, III (3.6)].
\par
For partitions $\la, \nu^{(1)}, \dots, \nu^{(r)}$, we 
define an integer $c^{\la}_{\nu^{(1)}, \dots, \nu^{(r)}}$ by 

\begin{equation*}
s_{\nu^{(1)}}\cdots s_{\nu^{(r)}} = \sum_{\la}c^{\la}_{\nu^{(1)},\dots, \nu^{(r)}}s_{\la}.
\end{equation*} 
In the case where $r = 2$, $c^{\la}_{\nu^{(1)},\nu^{(2)}}$ coincides with the Littlewood-Richardson 
coefficient. 
\par
For $\Bla \in \SP_{n,r}$, put 
\begin{equation*}
\tag{3.8.3}
b(\Bla) = a(\Bla) - r\cdot n(\Bla) = |\la^{(2)}| + 2|\la^{(3)}| + \cdots + (r-1)|\la^{(r)}|.
\end{equation*}
The following lemma is a generalization of [LS, Lemma 3.4]. 

\begin{lem}  %%%%% Lemma 3.9.
Let $\Bla, \Bmu \in \SP_{n,r}$, and assume that 
$\Bmu = (-, \dots, -, \xi)$.  Then we have

\begin{align*}
\tag{3.9.1}
K^-_{\Bla, \Bmu}(t) &= t^{b(\Bmu) - b(\Bla)}
          \sum_{\Bnu \trleq \Bla}
                       f^{\xi}_{\nu^{(1)}, \dots, \nu^{(r)}}(t^{r})
                          K_{\la^{(1)}, \nu^{(1)}}(t^r)\cdots K_{\la^{(r)}, \nu^{(r)}}(t^r), \\ 
\tag{3.9.2}
K^-_{\Bla, \Bmu}(t) &= 
        t^{b(\Bmu) - b(\Bla)}
             \sum_{\e \in \SP_n}c^{\eta}_{\la^{(1)}, \dots, \la^{(r)}}K_{\eta, \xi}(t^r). 
\end{align*}
\end{lem}

\begin{proof}
The formula (3.7.1) can be rewritten as 

\begin{equation*}
\tag{3.9.3}
K^-_{\Bla, \Bmu}(t) = t^{a(\Bmu) - a(\Bla) + rn(\Bla)}\sum_{\Bnu \trleq \Bla}t^{-rn(\Bnu)}
                       g^{\xi}_{\nu^{(1)}, \dots, \nu^{(r)}}(t^{-r})
                          K_{\la^{(1)}, \nu^{(1)}}(t^r)\cdots K_{\la^{(r)}, \nu^{(r)}}(t^r). 
\end{equation*}
Substituting (3.8.2) into (3.9.3), we obtain (3.9.1). 
Next we show (3.9.2). 
One can write as 

\begin{equation*}
s_{\la^{(i)}}(y) = \sum_{\nu^{(i)}}K_{\la^{(i)}, \nu^{(i)}}(t)P_{\nu^{(i)}}(y;t).
\end{equation*}
Hence 

\begin{align*}
\tag{3.9.4}
s_{\la^{(1)}}(y)\cdots s_{\la^{(r)}}(y) &= \sum_{\Bnu \in \SP_{n,r}}
    K_{\la^{(1)}, \nu^{(1)}}(t)\cdots K_{\la^{(r)}, \nu^{(r)}}(t)
         P_{\nu^{(1)}}(y;t)\cdots P_{\nu^{(r)}}(y;t) \\
    &= \sum_{\Bnu \in \SP_{n,r}}\sum_{\xi \in \SP_n}f^{\xi}_{\nu^{(1)}, \dots, \nu^{(r)}}(t)
                    K_{\la^{(1)},\nu^{(1)}}(t)\cdots K_{\la^{(r)},\nu^{(r)}}(t)P_{\xi}(y;t).
\end{align*}
On the other hand, 

\begin{align*}
\tag{3.9.5}
s_{\la^{(1)}}(y)\cdots s_{\la^{(r)}}(y) &= \sum_{\e \in \SP_n}
       c^{\eta}_{\la^{(1)}, \dots, \la^{(r)}}s_{\eta}(y)  \\
         &= \sum_{\eta \in \SP_n}c^{\eta}_{\la^{(1)}, \dots, \la^{(r)}}
                \sum_{\xi \in \SP_n} K_{\eta, \xi}(t)P_{\xi}(y;t).
\end{align*}
By comparing (3.9.4) and (3.9.5), we have an equality for each $\xi \in \SP_n$,  

\begin{equation*}
\sum_{\e \in \SP_n}c^{\eta}_{\la^{(1)}, \dots, \la^{(r)}}K_{\eta, \xi}(t) 
                    = \sum_{\Bnu \in \SP_{n,r}}f^{\xi}_{\nu^{(1)}, \dots, \nu^{(r)}}(t)
                          K_{\la^{(1)}, \nu^{(1)}}(t), \dots K_{\la^{(r)}, \nu^{(r)}}(t).
\end{equation*}
Combining this with (3.9.1), we obtain (3.9.2). The lemma is proved.
\end{proof}

\para{3.10.}
Let $\e' = \la' - \th', \e'' = \la'' - \th''$ be skew diagrams, where 
$\th' \subset \la', \th'' \subset \la''$ are partitions. 
We define a new skew diagram $\e'*\e'' = \la - \th$ as follows; 
write the partitions $\la', \la''$ as 
$\la' = (\la'_1, \dots, \la'_{k'}), \la'' = (\la''_1, \dots, \la''_{k''})$ 
with $\la'_{k'} > 0, \la''_{k''} > 0$. 
Put $a = \la_1''$.  We define a partition $\la = (\la_1, \dots, \la_{k' + k''})$ 
by 

\begin{equation*}
\la_i = \begin{cases}
          \la'_i + a  &\quad\text{ for } 1 \le i \le k', \\
           \la''_{i-k'}  &\quad\text{ for } k' + 1 \le i \le k' + k''.
        \end{cases}
\end{equation*}
Write partitions $\th', \th''$ as $\th' = (\th'_1, \dots, \th'_{k'}), 
\th'' = (\th''_1, \dots, \th''_{k''})$ with $\th'_{k'} \ge 0$, 
$\th''_{k''} \ge 0$.  
We define a partition $\th = (\th_1, \dots, \th_{k' + k''})$, in a similar 
way as above,  by 

\begin{equation*}
\th_i = \begin{cases}
          \th'_i + a  &\quad\text{ for } 1 \le i \le k', \\
           \th''_{i-k'}  &\quad\text{ for } k' + 1 \le i \le k' + k''.
        \end{cases}
\end{equation*}
We have $\th \subset \la$, and the skew  diagram  $\e'*\e'' = \la - \th$ 
can be defined. 
\par
For $\la, \mu \in \SP_n$, let $SST(\la, \mu)$ be the set of semistandard tableaux of 
shape $\la$ and weight $\mu$. 
Let $\Bla \in \SP_{n,r}$.  An $r$-tuple $T = (T^{(1)}, \dots, T^{(r)})$ 
is called a semistandard tableau of shape $\Bla$ if $T^{(i)}$ is a semistandard 
tableau of shape $\la^{(i)}$ with respect to the letters $\{ 1, \dots, n\}$. 
We denote by $SST(\Bla)$ the set of semistandard tableaux of shape $\Bla$. 
For $\Bla \in \SP_{n,r}$, let $\wt\Bla$ be the skew diagram 
$\la^{(1)}*\la^{(2)}*\cdots *\la^{(r)}$.  
Then $T \in SST(\Bla)$ is regarded as a usual semistandard tableau $\wt T$ 
associated to the skew diagram $\wt\Bla$. 
Assume $\pi \in \SP_n$.  We say that $T \in SST(\Bla)$ has weight $\pi$ if the 
corresponidng tableau $\wt T$ has shape $\wt\Bla$ and weight $\pi$. 
We denote by $SST(\Bla, \pi)$ the set of semistandard tableaux of shape $\Bla$ and 
weight $\pi$. 

\para{3.11.}
In [M, I, (9.4)], a bijective map  $\vT$

\begin{equation*}
\tag{3.11.1}
\vT : SST(\wt\Bla, \pi) \isom \coprod_{\nu \in \SP_n}(SST^0(\wt\Bla, \nu) \times SST(\nu, \pi))
\end{equation*}
was constructed, where $SST^0(\wt\Bla, \nu)$ is the set of tableau $T$ such that 
the associated word $w(T)$ is a lattice permutation (see [M, I, 9] for the definition).
Under the identification $SST(\wt\Bla, \pi) \simeq SST(\Bla, \pi)$, 
the subset $SST^0(\Bla,\nu)$ of $SST(\Bla, \nu)$ is also defined. 
Then we can regard $\vT$ as a bijection with respect to the set $SST(\Bla, \pi)$ 
(and $SST^0(\Bla, \nu)$).
\par
In the case where $r = 2$, it is shown in [LS, Cor. 3.9] that $|SST^0(\Bla, \nu)|$ 
coincides with the Littlewood-Richardson coefficient $c^{\nu}_{\la^{(1)}, \la^{(2)}}$. 
A similar argument can be applied also to the general case, and we have

\begin{cor}  %%%%  Corollary 3.12.
Assume that $\Bla \in \SP_{n,r}, \nu \in \SP_n$. Then we have

\begin{equation*}
|SST^0(\Bla, \nu)| = c^{\nu}_{\la^{(1)}, \dots, \la^{(r)}}.
\end{equation*}
\end{cor}

\para{3.13.}
For a semistandard tableau $S$, the charge $c(S)$ is defined as in [M, III, 6].
It is known that Lascoux-Sch\"utzenberger Theorem  
([M, III, (6.5)]) gives a combinatorial description of 
Koskta polynomials $K_{\la,\mu}(t)$ in terms of sesmistandard tableaux, 

\begin{equation*}
\tag{3.13.1}
K_{\la,\mu}(t) = \sum_{S \in SST(\la, \mu)}t^{c(S)}.
\end{equation*}

In the case where $r = 2$, a similar formula was proved for $K_{\Bla, \Bmu}(t)$ 
in [LS, Thm. 3.12], in the special case where $\Bmu = (-,\mu'')$.  
Here we consider $K_{\Bla, \Bmu}(t)$ for general $r$. 
Assume that $\Bla \in \SP_{n,r}$ and $\xi \in \SP_n$.  For $T \in SST(\Bla, \xi)$,
we write $\vT(T) = (D, S)$ with $S \in SST(\nu, \xi)$ for some $\nu$.    
we define a charge $c(T)$ of $T$ by $c(T) = c(S)$.  
We have the following theorem.  Note that the proof is quite similar to 
[LS]. 

\begin{thm}  %%%%  Theorem 3.14.
Let $\Bla, \Bmu \in \SP_{n,r}$, and assume that $\Bmu = (-, \dots, -, \xi)$. Then 

\begin{equation*}
K^-_{\Bla,\Bmu}(t) = t^{b(\Bmu) - b(\Bla)}\sum_{T \in SST(\Bla, \xi)}t^{r\cdot c(T)}.
\end{equation*}
\end{thm}

\begin{proof}
We define a map $\Psi : SST(\Bla, \xi) \to \coprod_{\nu \in \SP_n}SST(\nu,\xi)$ 
by $T \mapsto S$, where $\vT(T) = (D, S)$. Then by Corollary 3.12, for each 
$S \in SST(\nu, \xi)$, the set $\Psi\iv(S)$ has the cardinality 
$c^{\xi}_{\la^{(1)}, \dots, \la^{(r)}}$, and by definition, any $T \in \Psi\iv(S)$ has the
charge $c(T) = c(S)$.  Hence

\begin{align*}
\sum_{T \in SST(\Bla, \xi)}t^{c(T)} &= \sum_{\nu \in \SP_n}
                                       \sum_{S \in SST(\nu, \xi)}
          c^{\nu}_{\la^{(1)}, \dots, \la^{(r)}}t^{c(S)}  \\
             &= \sum_{\nu \in \SP_n}c^{\nu}_{\la^{(1)}, \dots, \la^{(r)}}K_{\nu, \xi}(t).
\end{align*}   
The last equality follows from (3.13.1). 
The theorem now follows from  (3.9.2). 
\end{proof}

\begin{cor}  %%%%%  Cor. 3.15.
Under the assumption of Theorem 3.14, we have

\begin{equation*}
K^-_{\Bla, \Bmu}(1) = |SST(\Bla, \xi)|.
\end{equation*}
\end{cor}

\para{3.16.}
In the rest of this section, we shall give an alternate description 
of the polynomial $g^{\Bmu}_{\Bnu}(t)$ in the case where 
$\Bmu = (-, \dots, -, \xi)$. 
For $\Bnu \in \SP_{n,r}$, put 
$R_{\Bnu}(x;t) = P_{\nu^{(1)}}(x^{(1)};t^r)\cdots P_{\nu^{(r)}}(x^{(r)};t^r)$.
Then $\{ R_{\Bnu} \mid \Bnu \in \SP_{n,r} \}$ gives a basis of $\Xi^n[t]$. 
We define funtions $h^{\Bmu}_{\Bnu}(t) \in \BQ(t)$ by the condition that

\begin{equation*}
\tag{3.16.1}
R_{\Bnu}(x;t) = \sum_{\Bmu \in \SP_{n,r}}h^{\Bmu}_{\Bnu}(t)P^-_{\Bmu}(x;t).
\end{equation*}

The following formula is a generalization of Proposition 4.2 in [LS].

\begin{prop}
Assume that $\Bmu = (-,\dots,-,\xi)$. Then 

\begin{equation*}
h^{\Bmu}_{\Bnu}(t) = t^{a(\Bmu) - a(\Bnu)}g^{\Bmu}_{\Bnu}(t^{-r}).
\end{equation*}
\end{prop}

\begin{proof}
The proof is quite similar to that of [LS, Prop. 4.2].  
For $\Bla \in \SP_{n,r}$, we have

\begin{align*}
s_{\Bla}(x) &= s_{\la^{(1)}}(x^{(1)})\cdots s_{\la^{(r)}}(x^{(r)}) \\
            &= \prod_{i=1}^r\sum_{\nu^{(i)}}K_{\la^{(i)},\nu^{(i)}}(t^r)P_{\nu^{(i)}}(x^{(i)};t^r)  \\
            &= \sum_{\Bnu}K_{\la^{(1)},\nu^{(1)}}(t^r)\cdots K_{\la^{(r)},\nu^{(r)}}(t^r)
                  \sum_{\Bmu \in \SP_{n,r}}h^{\Bmu}_{\Bnu}(t)P^-_{\Bmu}(x;t) \\
            &= \sum_{\Bmu \in \SP_{n,r}}\biggl(
                 \sum_{\Bnu}K_{\la^{(1)},\nu^{(1)}}(t^r)\cdots K_{\la^{(r)}, \nu^{(r)}}(t^r)
                    h^{\Bmu}_{\Bnu}(t)\biggr)P^-_{\Bmu}(x;t).
\end{align*}
Since $s_{\Bla}(x) = \sum_{\Bmu \in \SP_{n,r}}K^-_{\Bla,\Bmu}(t)P_{\Bmu}^-(x;t)$, 
by comparing the coefficients of $P^-_{\Bmu}(x;t)$, we have

\begin{equation*}
\tag{3.17.1}
K^-_{\Bla,\Bmu}(t) = \sum_{\Bnu \in \SP_{n,r}}h^{\Bmu}_{\Bnu}(t)
           K_{\la^{(1)}, \nu^{(1)}}(t^r)\cdots K_{\la^{(r)}, \nu^{(r)}}(t^r).
\end{equation*}

Now assume that $\Bmu = (-,\dots,-,\xi)$. If we notice that 
$K_{\la^{(i)},\nu^{(i)}}(t^r) \ne 0$ only when $|\la^{(i)}| = |\nu^{(i)}|$, 
(3.9.3) implies that 

\begin{equation*}
\tag{3.17.2}
K^-_{\Bla,\Bmu}(t) = \sum_{\Bnu  \in \SP_{n,r}}t^{a(\Bmu) - a(\Bnu)}g^{\Bmu}_{\Bnu}(t^{-r})
        K_{\la^{(1)}, \nu^{(1)}}(t^r)\cdots K_{\la^{(r)},\nu^{(r)}}(t^r).
\end{equation*}  
Since $(K_{\la^{(1)},\nu^{(1)}}(t^r)\cdots K_{\la^{(r)},\nu^{(r)}}(t^r))_{\Bla, \Bnu \in \SP_{n,r}}$
is a unitriangular matrix, the proposition follows by comparing (3.17.1) and (3.17.2).
\end{proof}

\par\vspace{1cm}
\noindent
T. Shoji \\
Department of Mathematics, Tongji University \\ 
1239 Siping Road, Shanghai 200092, P. R. China  \\
E-mail: \verb|shoji@tongji.edu.cn|

\end{document}